\documentclass[]{amsart}

\usepackage{amsbsy, amscd, amsmath, amssymb, amstext, amsxtra, commath, enumitem, eucal, float, graphicx, hyperref, latexsym, mathrsfs, mathtools, multicol, tikz, tikz-cd}
\usepackage[utf8]{inputenc}
\usepackage[all]{xy}

\usepackage{lineno}

\newtheorem{theorem}{Theorem}[section]
\newtheorem{lemma}[theorem]{Lemma}
\newtheorem{proposition}[theorem]{Proposition}
\newtheorem{corollary}[theorem]{Corollary}
\theoremstyle{definition}

\newtheorem{remark}[theorem]{Remark}

\numberwithin{equation}{section}

\newtheorem{question}[theorem]{Question}

\def\C{\mathop{\mathsf{C}\hspace{0mm}}\nolimits}
\def\CN{\mathop{\mathsf{CN}}\nolimits}
\def\P{\mathop{\mathsf{P}\hspace{0mm}}\nolimits}

\def\UC{\mathop{\mathsf{UC}}\nolimits}
\def\WP{\mathop{\mathsf{WP}}\nolimits}

\def\cf{\mathop{\rm cf}\nolimits}

\def\CL{\mathop{\rm CL}\nolimits}
\def\cof{\mathop{\rm cof}\nolimits}
\def\cov{\mathop{\rm cov}\nolimits}

\def\int{\mathop{\rm int}\nolimits}

\def\F{\mathop{\mathcal{F}\hspace{0mm}}\nolimits}
\def\HH{\mathop{\mathcal{H}\hspace{0mm}}\nolimits}
\def\K{\mathop{\mathcal{K}\hspace{0mm}}\nolimits}

\def\Pot{\mathop{\mathscr{P}\hspace{0mm}}\nolimits}

\newcommand{\vt}[1]{\langle #1\rangle}

\let\mathcal\mathscr

\tikzcdset{scale cd/.style={every label/.append style={scale=#1},
    cells={nodes={scale=#1}}}}
    
    \newcommand{\stickT}{\setbox255=\hbox{\raise1ex\hbox{$\hspace{0.2pt}\,\bullet\,$}}\mathord{\rlap{\hbox to\wd255{\hss\hbox{$|$}\hss}}\box255}}
\newcommand{\stickS}{\setbox255=\hbox{\raise0.6ex\hbox{$\scriptstyle\bullet$}}\mathord{\rlap{\hbox to\wd255{\hss\hbox{$\scriptstyle|$}\hss}}\box255}}

\begin{document}



%\linenumbers

\title{Calibers of canonical hyperconnected spaces and their Vietoris hyperspaces}

\author{Alejandro R\'ios Herrej\'on}
\address {A. Ríos Herrej\'{o}n\\
Departamento de Matem\'aticas, Facultad de Ciencias, Universidad Nacional Aut\'onoma de M\'exico, Circuito ext. s/n, Ciudad Universitaria, C.P. 04510,  M\'exico, CDMX}
\email{chanchito@ciencias.unam.mx}

\author[\'A. Tamariz Mascar\'ua]{\'Angel Tamariz Mascar\'ua}
\address{Á. Tamariz Mascar\'ua, Departamento de Matem\'aticas, Facultad de Ciencias, Universidad Nacional Aut\'onoma de M\'exico, Circuito ext. s/n, Ciudad Universitaria, C.P. 04510,  M\'exico, CDMX}
\email{atamariz@unam.mx}
\urladdr{}

\begin{abstract}
  For infinite cardinals $\lambda \leq \kappa$, 
we calculate the calibers of topological spaces $X$ of cardinality $\kappa$ 
whose open sets are the subsets $A$ satisfying $|X \setminus A| < \lambda$. 
We also prove that the set of calibers of 
$X$ coincides with the set of calibers of the Vietoris hyperspace $\mathcal{H}(X)$, where the space $\mathcal{H}(X)$ has as its underlying set a collection of subsets of $X$ that includes all its finite subsets and is contained within the collection of its closed subsets.
\end{abstract}

\maketitle

%\tableofcontents

%%%%%%%%%%%%%%%%%%%%%%%%%%%%%%%%%%%%%%%%%%%%%%%%%%%%%%%%%%%%%%
\section{Introduction}
%%%%%%%%%%%%%%%%%%%%%%%%%%%%%%%%%%%%%%%%%%%%%%%%%%%%%%%%%%%%%%

A cardinal number $\kappa$ is a {\it caliber} of the topological space $(X,\tau)$ if for every family of nonempty open subsets of $X$ indexed by $\kappa$,
$\mathcal{U} = \{U_\alpha \in \tau \setminus \{\emptyset\} : \alpha < \kappa\}$, there exists $J(\mathcal{U}) \subseteq \kappa$ such that $|J(\mathcal{U})| = \kappa$ and
$\bigcap_{\alpha \in J(\mathcal{U})}U_\alpha \neq \emptyset$. We denote the 
collection of calibers of a space $X$ as $\C(X)$.

The extensive collection of results, techniques, and relations between various mathematical notions surrounding the problem of determining the calibers of a topological space, as accomplished up to the 1980s, was compiled in {\it Chain conditions in topology} by Comfort and Negrepontis published by Cambridge University Press in 1982 \cite{comneg1982}. \v{S}anin initiated the study of calibers in his works \cite{sanin1946}, \cite{sanin1946_2}, \cite{sanin1948}, and \cite{sanin1948_2}.
In \cite{sanin1948_2}, \v{S}anin determined the very important relation between the calibers of a product of spaces and those of its factors.

Other authors who have contributed to the development of the theory of calibers include Argyros and Tsarpalias \cite{argtsa1982} (calibers of compact Hausdorff spaces), Arhangel'ski\u{\i} and Tkachuk \cite{arhtka1986} (calibers of $C_p(X)$), Juhász and Shelah \cite{juhshe2003} (calibers and left-separated spaces), Shakhmatov \cite{shakhmatov1986} (precalibers of $\sigma$-compact topological groups), and Tall \cite{tall1977} (first countable spaces with caliber $\aleph_1$).
The authors of the present text have also contributed to this matter; see, for example, \cite{rios2022} (singular precalibers for topological products) and \cite{riotam2023} (a comparison between the definition of caliber given by Engelking in \cite{engelking1989} and the traditional definition used in this article).
Results concerning the calibers of topological groups can be found in the text \cite{arhtka2008} by Arkhangel’skii and Tkachenko, while calibers in spaces of continuous functions of the type $C_p(X)$ have been studied by Kalamidas and Spiliopoulos \cite{kalspi1992} (relation of chain conditions on a space
$X$, with the weights of compact sets in $C_p(X)$), and by Tkachuk \cite{tkachuk1986} (calibers of $C_p(X)$).

\medskip

In this article we determine all the calibers of spaces $X$ of cardinality $\kappa$ with the topology $\tau = \{A \subseteq X : |X \setminus A| < \lambda\}$, where $\lambda$ and $\kappa$ are infinite cardinal numbers and $\lambda \leq \kappa$. 
We denote these spaces as $X(\lambda,\kappa)$ and call them {\it canonical hyperconnected spaces}, since they are the simplest hyperconnected spaces. Furthermore, we obtain interesting relations between the set of calibers of $X(\lambda,\kappa)$ and the set of calibers of the hyperspace $\mathcal{H}(X(\lambda,\kappa))$, where $\mathcal{H}(X(\lambda,\kappa))$ is any hyperspace of subsets of $X(\lambda,\kappa)$ that contains the collection of its finite subsets, is contained within the collection of its closed subsets, and is equipped with the Vietoris topology.

In Section 2, we provide the basic definitions and known results upon which this research relies. In Section 3, we determine the pseudocharacter and the $\pi$-weight of the canonical hyperconnected spaces; these cardinal topological functions play a fundamental role in calculating the calibers of these spaces. We determine the calibers of $X(\lambda,\kappa)$ in Section 4; and in Section 5 we prove several propositions that lead us to our main result (see Theorem~\ref{thm_hiphip_calibres_resumen}):

\smallskip

\noindent {\bf Theorem.} {\it For two infinite cardinals $\lambda \leq \kappa$, 
$\C(X(\lambda,\kappa)) = \C(\mathcal{H}(X(\lambda,\kappa)))$, where $\mathcal{H}(X(\lambda,\kappa))$ is any collection of subsets of $X(\lambda,\kappa)$ that contains all of its finite subsets, is contained in the collection of its closed subsets, and equipped with the Vietoris topology.} 

%%%%%%%%%%%%%%%%%%%%%%%%%%%%%%%%%%%%%%%%%%%%%%%%%%%%%%%%%%%%
\section{Preliminaries}
%%%%%%%%%%%%%%%%%%%%%%%%%%%%%%%%%%%%%%%%%%%%%%%%%%%%%%%%%%%%

The purpose of this section is to specify the terminology and notation used throughout the article. In general, we follow the standard conventions in the literature: for general topology, we refer to \cite{engelking1989}, and for set theory, to \cite{kunen1980}.

The symbol $\omega$ denotes both the first infinite ordinal and the first infinite cardinal. On the other hand, we define $\mathbb{N} := \omega \setminus \{0\}$. Also, $\mathfrak{c}$ denotes the cardinal number $2^\omega$, that is, the cardinality of the continuum. Likewise, for any set $X$, $\mathcal{P}(X)$ denotes the \emph{power set} of $X$.

Let $X$ be a topological space. The symbol $\tau_X$ denotes the topology of $X$, while $\tau_X^+$ denotes the set $\tau_X \setminus \{\emptyset\}$.

We say that an infinite cardinal $\kappa$ is a:

\begin{enumerate}
\item \emph{caliber} for $X$ if, for every $\{U_\alpha : \alpha < \kappa\} \subseteq \tau_X^+$, there exists $J \in [\kappa]^{\kappa}$ such that $\{U_\alpha : \alpha \in J\}$ has a nonempty intersection.
\item \emph{precaliber} for $X$ if, for every $\{U_\alpha : \alpha < \kappa\} \subseteq \tau_X^+$, there exists $J \in [\kappa]^{\kappa}$ such that $\{U_\alpha : \alpha \in J\}$ is a centered family.
\item \emph{weak precaliber} for $X$ if, for every $\{U_\alpha : \alpha < \kappa\} \subseteq \tau_X^+$, there exists $J \in [\kappa]^{\kappa}$ such that $\{U_\alpha : \alpha \in J\}$ is a linked family.
\end{enumerate}

In turn, the symbols $\C(X)$, $\P(X)$, and $\WP(X)$ denote, respectively, the collections of calibers, precalibers, and weak precalibers of the space $X$. Similarly, $\CN$ and $\UC$ denote, respectively, the class of infinite cardinals and the class of cardinals of uncountable cofinality.

Proposition~\ref{prop_basico} collects several results concerning the preceding concepts that are used throughout the article. Items~(\ref{basico_cofinalidad})--(\ref{basico_producto}) are classical results from the theory of calibers (see \cite{comneg1982}), while the proofs of items~(\ref{basico_T1})--(\ref{basico_hiperconexo}) can be found in \cite{riotam2023}.

Recall that a \emph{$\pi$-base} for a topological space $X$ is a family $\mathcal{V}\subseteq\tau_X^+$ such that, for every $U\in\tau_X^+$, there exists $V\in\mathcal{V}$ with $V\subseteq U$. The \emph{$\pi$-weight} of $X$ is the cardinal number
\[
\pi w(X)
:= \min\{|\mathcal{V}| : \mathcal{V} \text{ is a $\pi$-base for } X\}.
\]
Likewise, we say that $X$ is \emph{hyperconnected} if, for any $U,V \in \tau_X^+$, we have $U \cap V \neq \emptyset$.

\begin{proposition}\label{prop_basico}
The following statements hold for every topological space $X$ and every infinite cardinal $\kappa$.

\begin{enumerate}
\item\label{basico_cofinalidad} If $\kappa$ is a caliber for $X$, then $\cf(\kappa)$ is also a caliber for $X$.

\item\label{basico_densidad} If $\cf(\kappa) > d(X)$, then $\kappa$ is a caliber for $X$.

\item\label{basico_pi_peso} If $\kappa > \pi w(X)$ and $X$ has caliber $\cf(\kappa)$, then $\kappa$ is a caliber for $X$.

\item\label{basico_denso} If $D$ is a dense subset of $X$, then $\WP(D) = \WP(X)$, $\P(D) = \P(X)$, and $\C(D) \subseteq \C(X)$.

\item\label{basico_funcion} If a topological space $Y$ is a continuous image of $X$, then $\C(X) \subseteq \C(Y)$.

\item\label{basico_T2} If $X$ is infinite and $T_2$, then $\WP(X) \subseteq \UC$.

\item\label{basico_producto} The property of having caliber $\kappa$ is finitely productive.

\item\label{basico_T1} If $X$ is countably infinite and $T_1$, then $\C(X) = \UC$.

\item\label{basico_cofinito} If $X$ is uncountable and cofinite, then $\C(X)=\P(X)=\WP(X)=\CN$.

\item\label{basico_hiperconexo} If $X$ is hyperconnected, then $\P(X)=\WP(X)=\CN$.

\end{enumerate}

\end{proposition}

From now on, we work under the axioms of ${\sf ZFC}$, that is, the Zermelo--Fraenkel axioms together with the Axiom of Choice. Any additional axiom required in a statement or proof will be explicitly indicated.

The Greek letters $\kappa$, $\lambda$, $\mu$, $\nu$, and $\theta$ denote cardinal numbers. Similarly, the symbols $\alpha$, $\beta$, $\gamma$, $\delta$, $\xi$, and $\eta$ denote ordinal numbers. Likewise, the symbols $\kappa^+$ and $\cf(\kappa)$ denote the successor cardinal of $\kappa$ and the cofinality of $\kappa$, respectively.

Throughout this text, for a given set $X$, the symbols $[X]^{\leq \kappa}$, $[X]^{<\kappa}$, and $[X]^{\kappa}$ denote, respectively, the subcollections of $\mathcal{P}(X)$ consisting of the subsets of $X$ of cardinality at most $\kappa$, the subsets of $X$ of cardinality strictly less than $\kappa$, and the subsets of $X$ of cardinality exactly $\kappa$.

The central theme of this work is the class of spaces described below. Suppose that the cardinals $\lambda$ and $\kappa$ satisfy $\omega \leq \lambda \leq \kappa$. If $X$ is a set of cardinality $\kappa$, then the collection
\[
\tau := \{A \subseteq X : |X \setminus A| < \lambda\} \cup \{\emptyset\}
= \{X \setminus A : A \in [X]^{<\lambda}\} \cup \{\emptyset\}
\]
defines a $T_1$, crowded, homogeneous, and hyperconnected topology on $X$. %To verify the last property, observe that if $A$ and $B$ are nonempty elements of $\tau$ and $A \cap B = \emptyset$, then the equality $X = (X \setminus A) \cup (X \setminus B)$ implies that $|X| < \lambda \leq \kappa$, which is impossible.

In the context of the preceding paragraph, the space $X$ is denoted by $X(\lambda,\kappa)$, and the topology described above by $\tau(\lambda,\kappa)$. Moreover, throughout this article, any space of this type is called a \emph{canonical hyperconnected space}. When $\lambda$ coincides with $\kappa$, we write simply $X(\kappa)$ and $\tau(\kappa)$ instead of $X(\lambda,\kappa)$ and $\tau(\lambda,\kappa)$, respectively.

We conclude the section with one final comment. In the subsequent results, in order to simplify the notation and facilitate the exposition, statements are formulated in terms of the spaces $X(\lambda,\kappa)$, as appropriate. In the proofs, however, we refer only to the underlying set $X$, in order to avoid cumbersome notation resulting from the constant repetition of the symbol $X(\lambda,\kappa)$.

%%%%%%%%%%%%%%%%%%%%%%%%%%%%%%%%%%%%%%%%%%%%%%%%%%%%%%%%%%%%%%
\section{Topological Cardinal Functions}\label{secc_FCT}
%%%%%%%%%%%%%%%%%%%%%%%%%%%%%%%%%%%%%%%%%%%%%%%%%%%%%%%%%%%%%%

In this section, we carry out a detailed study of the cardinal functions of the spaces $X(\lambda,\kappa)$, with the aim of using these calculations later to determine their chain conditions, with particular emphasis on obtaining their calibers.

The first cardinal functions that we analyze in the context of canonical hyperconnected spaces are density and hereditary density:

\begin{proposition}\label{prop_FCT_densidad}
We have
\[
\textstyle d(X(\lambda,\kappa)) = \lambda = hd(X(\lambda,\kappa)).
\]

\end{proposition}

\begin{proof} On the one hand, every subset of $X$ of cardinality $\lambda$ is dense in $X$, which implies that $d(X) \leq \lambda$. On the other hand, if $A \in [X]^{<\lambda}$, then $X\setminus A$ is a nonempty open subset of $X$ that does not intersect $A$; in particular, $A$ is not dense in $X$. Thus, $d(X)=\lambda$.

For hereditary density, let $Y$ be a subset of $X$. If $|Y|<\lambda$, then $Y$ inherits the discrete topology and, consequently, $d(Y)=|Y|$. On the other hand, if $|Y|\geq\lambda$, then $Y$ inherits the topology $\tau(\lambda,|Y|)$, and hence its density is $\lambda$.
\end{proof}

The next step is to analyze certain local cardinal functions. To avoid terminological ambiguities, we briefly summarize the terminology we use below.

Let $X$ be a topological space and let $x\in X$. A \emph{local $\pi$-base} at $x$ in $X$ is a collection $\mathcal{V}\subseteq\tau_X^+$ such that, for every $U\in\tau_X$ with $x\in U$, there exists $V\in\mathcal{V}$ such that $V\subseteq U$. If, in addition, $x\in\bigcap\mathcal{V}$, we say that $\mathcal{V}$ is a \emph{local base} at $x$ in $X$. On the other hand, when $X$ is a $T_1$ space, a \emph{pseudobase} at $x$ in $X$ is a collection $\mathcal{V}\subseteq\tau_X$ such that $\bigcap\mathcal{V}=\{x\}$.

The \emph{$\pi$-character} of $x$ in $X$, the \emph{character} of $x$ in $X$, and the \emph{pseudocharacter} of $x$ in $X$ are, respectively, the cardinal numbers
\begin{align*}
\pi\chi(x,X)
&:= \min\{|\mathcal{V}| : \mathcal{V} \text{ is a local $\pi$-base at } x \text{ in } X\}, \\
\chi(x,X)
&:= \min\{|\mathcal{V}| : \mathcal{V} \text{ is a local base at } x \text{ in } X\} \hspace{0.5em} \text{and} \\
\psi(x,X)
&:= \min\{|\mathcal{V}| : \mathcal{V} \text{ is a pseudobase at } x \text{ in } X\}.
\end{align*}

Likewise, the \emph{$\pi$-character}, the \emph{character}, and the \emph{pseudocharacter} of $X$ are, respectively, the cardinals
\begin{align*}
\pi\chi(X)
&:= \sup\{\pi\chi(x,X) : x\in X\}, \\
\chi(X)
&:= \sup\{\chi(x,X) : x\in X\} \hspace{0.5em} \text{and} \\
\psi(X)
&:= \sup\{\psi(x,X) : x\in X\}.
\end{align*}

\begin{proposition}\label{prop_FCT_pi_caracter} We have
\[
\pi\chi(x,X(\lambda,\kappa)) = \pi\chi(X(\lambda,\kappa)) = \pi w(X(\lambda,\kappa))
\]
for every $x \in X(\lambda,\kappa)$.

\end{proposition}

\begin{proof} First, since $X$ is homogeneous, for any $x,y \in X$ we have $\pi\chi(x,X) = \pi\chi(y,X)$. Thus, the collection $\{\pi\chi(z,X) : z \in X\}$ consists of a single element, and consequently $\pi\chi(x,X) = \pi\chi(X)$ for every $x \in X$.

On the other hand, we always have $\pi\chi(X) \leq \pi w(X)$. To establish the reverse inequality, it suffices to observe that if $x \in X$ and $\mathcal{V}$ is a local $\pi$-base at $x$ in $X$ such that $|\mathcal{V}| = \pi\chi(x,X)$, then $\mathcal{W} := \{V \setminus \{x\} : V \in \mathcal{V}\}$ is a $\pi$-base for $X$ of the same cardinality as $\mathcal{V}$. Indeed, if $U \in \tau_X^+$, then $U \cup \{x\}$ is an open subset of $X$ containing $x$, which implies that there exists $V \in \mathcal{V}$ such that $V \subseteq U \cup \{x\}$ and, consequently, $V \setminus \{x\} \subseteq U$.
\end{proof}

\begin{proposition}\label{prop_FCT_caracter} We have
\[
\chi(x,X(\lambda,\kappa)) = \chi(X(\lambda,\kappa)) = \pi w(X(\lambda,\kappa))
\]
for every $x \in X(\lambda,\kappa)$.

\end{proposition}

\begin{proof} Once again, since $X$ is homogeneous, for any $x,y\in X$ we have $\chi(x,X)=\chi(y,X)$. Consequently, the set $\{\chi(z,X): z\in X\}$ consists of a single element, and therefore $\chi(x,X)=\chi(X)$ for every $x\in X$.

To prove the equality $\chi(X)=\pi w(X)$, first recall that the inequality $\pi\chi(X)\leq\chi(X)$ always holds and that, moreover, the equality $\pi w(X)=\pi\chi(X)$ follows from Proposition~\ref{prop_FCT_pi_caracter}. Thus, it suffices to show that $\chi(X)\leq\pi w(X)$.

To this end, observe that if $x\in X$ and $\mathcal{V}$ is a $\pi$-base for $X$ of cardinality $\pi w(X)$, then $\{V\cup\{x\}:V\in\mathcal{V}\}$ is a local base at $x$ in $X$. Indeed, if $U\in\tau_X$ contains $x$, then $U\setminus\{x\}$ is a nonempty open subset of $X$. Since $\mathcal{V}$ is a $\pi$-base for $X$, there exists $V\in\mathcal{V}$ such that $V\subseteq U\setminus\{x\}$. Thus, $V\cup\{x\}\subseteq U$.
\end{proof}

\begin{proposition}\label{prop_FCT_pseudocaracter_1} We have
\[\textstyle
\psi(X(\lambda,\kappa)) = \min \big\{|\mathcal{V}| : \mathcal{V} \subseteq \tau(\lambda,\kappa)^+ \ \wedge \ \bigcap \mathcal{V} = \emptyset\big\}.
\]

\end{proposition}

\begin{proof} Let $\mu := \min \{|\mathcal{V}| : \mathcal{V} \subseteq \tau_X^+ \ \wedge \ \bigcap \mathcal{V} = \emptyset\}$. To prove the desired equality, it suffices to show that $\psi(x,X) = \mu$ for every $x \in X$.

With this in mind, fix an element $x$ of $X$. First, observe that if $\mathcal{V}$ is a pseudobase at $x$ in $X$, then the collection $\mathcal{W} := \{V \setminus \{x\} : V \in \mathcal{V}\}$ consists of elements of $\tau_X^+$, satisfies $\bigcap \mathcal{W} = \emptyset$, and has the same cardinality as $\mathcal{V}$. Conversely, if $\mathcal{V} \subseteq \tau_X^+$ satisfies $\bigcap \mathcal{V} = \emptyset$, then $\mathcal{W} := \{V \cup \{x\} : V \in \mathcal{V}\}$ is a pseudobase at $x$ in $X$ of the same cardinality as $\mathcal{V}$.
\end{proof}

\begin{remark}\label{rmk_FCT_cubierta_cofinalidad} A consequence of Proposition~\ref{prop_FCT_pseudocaracter_1} is that, if $\mathcal{I} := [X]^{<\lambda}$, then $\psi(X(\lambda,\kappa))$ coincides with the covering number of the ideal $\mathcal{I}$; that is,
\[
\textstyle
\psi(X(\lambda,\kappa)) = \cov(\mathcal{I}) = \min \{|\mathcal{A}| : \mathcal{A} \subseteq \mathcal{I} \ \wedge \ \bigcup \mathcal{A} = X\}.
\]

Similarly, $\pi w(X(\lambda,\kappa))$ coincides with the cofinality of the ideal $\mathcal{I}$; that is,
\[
\textstyle
\pi w(X(\lambda,\kappa)) = \cof(\mathcal{I}) = \min \{|\mathcal{A}| : \mathcal{A} \subseteq \mathcal{I} \ \wedge \ \forall B \in \mathcal{I} \ \exists A \in \mathcal{A} (B \subseteq A)\}.
\]

\end{remark}

Remark~\ref{rmk_FCT_cubierta_cofinalidad} allows us to determine the cardinal $\psi(X(\lambda,\kappa))$ precisely in terms of the relation between $\lambda$ and $\kappa$. Indeed, since the computation reduces to determining the covering number of $[X]^{<\lambda}$, and this cardinal is well known (namely, $\kappa$ when $\lambda<\kappa$ and $\cf(\kappa)$ when $\lambda=\kappa$), we obtain the following topological information about the space $X(\lambda,\kappa)$.

\begin{proposition}\label{prop_FCT_pseudocaracter_2} We have
\[
\psi(X(\lambda,\kappa)) = \begin{cases} \kappa, & \text{if} \ \lambda < \kappa, \\
\cf(\kappa), & \text{if} \ \lambda = \kappa.
  \end{cases}
\]

\end{proposition}

In fact, the perspective of studying the topological space $X(\lambda,\kappa)$ through the invariants associated with the ideal $[X]^{<\lambda}$ proves to be particularly fruitful.

\begin{proposition}\label{prop_FCT_pi_peso_cotas} We have
\[
\kappa \leq \pi w(X(\lambda,\kappa)) = w(X(\lambda,\kappa)) \leq \kappa^{<\lambda}.
\]

\end{proposition}

\begin{proof} First, since the cofinality of $\mathcal{I} := [X]^{<\lambda}$ is always greater than or equal to its covering number, Remark~\ref{rmk_FCT_cubierta_cofinalidad} and Proposition~\ref{prop_FCT_pseudocaracter_2} imply that
\[
\pi w(X)=\cof(\mathcal{I}) \geq \cov(\mathcal{I}) = \psi(X) = \kappa,
\]
whenever $\lambda < \kappa$. On the other hand, $\pi w(X)\geq d(X)=\kappa$ when $\lambda=\kappa$ (see Proposition~\ref{prop_FCT_densidad}).

To establish the equality $\pi w(X)=w(X)$, it suffices to apply Proposition~\ref{prop_FCT_caracter}, together with the inequality $\pi w(X)\geq\kappa$ obtained above and the classical relation between $\chi(X)$, $w(X)$, and $|X|$:
\[
\pi w(X)=\chi(X)\leq w(X)\leq \chi(X)\cdot |X|=\pi w(X)\cdot |X|=\pi w(X)\cdot \kappa=\pi w(X).
\]
Consequently, $\pi w(X)=w(X)$.

Finally, the inequality $w(X) \leq \kappa^{<\lambda}$ holds, since
\[
w(X) \leq o(X) = |\tau_X| = |\mathcal{I}| = \kappa^{<\lambda}.
\]
\end{proof}

However, in contrast to the relative simplicity of computing $\psi(X(\lambda,\kappa))$, determining the cardinal $\pi w(X(\lambda,\kappa))$ precisely becomes considerably more difficult when $\lambda > \omega$. In the case $\lambda=\omega$, the cardinal $\pi w(X(\lambda,\kappa))$ coincides with the $\pi$-weight of the cofinite space of cardinality $\kappa$, which is exactly $\kappa$.

Among the known results, it is known that $\cof([\mathfrak{c}]^{<\aleph_1})=\mathfrak{c}$ and that $\cof([\aleph_n]^{<\aleph_1})=\aleph_n$ for every $n<\omega$. On the other hand, if $0^{\#}$ does not exist and $\kappa$ has uncountable cofinality, then $\cof([\kappa]^{<\aleph_1})=\kappa$ (see \cite{barjud1995}).

Likewise, pcf theory is closely related to this problem. For example, in \cite[Theorem~5.12, p.~1198]{abrmag2010}, this tool is used to prove that $\cof([\aleph_\omega]^{<\aleph_1}) < \aleph_{\mathfrak{c}^+}$. However, to the best of the authors' knowledge, the precise value of $\cof([\aleph_\omega]^{<\aleph_1})$ is not known.

Regarding the lower bound for the $\pi$-weight established in Proposition~\ref{prop_FCT_pi_peso_cotas}, it can be slightly improved in certain special cases involving $\lambda$ and $\kappa$. One such case is presented below in Proposition~\ref{prop_FCT_pi_peso_singular}, and another in Proposition~\ref{prop_FCT_cf_kappa_menor_lambda}, which we prove next.

\begin{proposition}\label{prop_FCT_cf_kappa_menor_lambda} If $\cf(\kappa) < \lambda < \kappa$, then $\pi w(X(\lambda,\kappa)) \geq \kappa^+$.

\end{proposition}

\begin{proof} Let $\{\kappa_\xi : \xi < \cf(\kappa)\}$ be a strictly increasing sequence of cardinals such that
$\sup\{\kappa_\xi : \xi < \cf(\kappa)\} = \kappa$. In addition, let $\{A_\alpha : \alpha < \kappa\}$ be a subset of $[X]^{<\lambda}$. For each $\xi < \cf(\kappa)$, the fact that $\kappa_\xi < \kappa = \psi(X)$ (see Proposition~\ref{prop_FCT_pseudocaracter_2}) implies that $\{A_\alpha : \alpha < \kappa_\xi\}$ is not a cover of $X$ and, hence, there exists $x_\xi \in X \setminus \bigcup_{\alpha < \kappa_\xi} A_\alpha$. Thus, the set $A := \{x_\xi : \xi < \cf(\kappa)\}$ is an element of $[X]^{<\lambda}$ such that $A \not\subseteq A_\alpha$ for every $\alpha < \kappa$; in particular, $\{A_\alpha : \alpha < \kappa\}$ is not cofinal in $[X]^{<\lambda}$.
\end{proof}

Despite the difficulties associated with the $\pi$-weight of the spaces $X(\lambda,\kappa)$, it is possible to determine this cardinal function precisely when $\lambda = \kappa$ and $\kappa$ is a regular cardinal.

\begin{proposition}\label{prop_FCT_kappa_regular} If $\kappa$ is regular and $\varphi \in \{hd, d, \psi, \chi, \pi\chi, \pi w, w\}$, then
\[
\varphi(X(\kappa)) = \kappa.
\]

\end{proposition}

\begin{proof} First, by Propositions~\ref{prop_FCT_densidad}--\ref{prop_FCT_caracter}, \ref{prop_FCT_pseudocaracter_2}, and \ref{prop_FCT_pi_peso_cotas}, as well as the classical relations among the local cardinal functions, we have
\[
hd(X) = d(X) = \kappa
= \psi(X)
\leq \chi(X)
= \pi\chi(X)
= \pi w(X) = w(X).
\]
Therefore, for our purposes, it suffices to show that the ideal $[X]^{<\kappa}$ has a cofinal subset of cardinality $\kappa$.

Let $\{x_\alpha : \alpha<\kappa\}$ be an enumeration without repetitions of $X$. Moreover, for each $\beta<\kappa$, let $A_\beta := \{x_\alpha : \alpha<\beta\}$. Clearly, $\{A_\beta : \beta<\kappa\}$ is a subset of $[X]^{<\kappa}$. Furthermore, if $B \in [X]^{<\kappa}$ and $J := \{\alpha<\kappa : x_\alpha \in B\}$, then $J$ has cardinality less than $\kappa$, which implies that $J$ is bounded in $\kappa$. Hence, there exists $\beta<\kappa$ such that $J \subseteq \beta$; consequently, $B \subseteq A_\beta$.
\end{proof}

In the case of the $\pi$-weight of the space $X(\kappa)$, when $\kappa$ is a singular cardinal, obtaining a precise computation is difficult in \textsf{ZFC}; nevertheless, it is possible to refine the lower bound established in Proposition~\ref{prop_FCT_pi_peso_cotas}.

\begin{proposition}\label{prop_FCT_pi_peso_singular} If $\kappa$ is singular, then $\pi w(X(\kappa)) \geq \kappa^+$.

\end{proposition}

\begin{proof} Let $\{\kappa_\xi : \xi < \cf(\kappa)\}$ be a strictly increasing sequence of cardinals such that
$\sup\{\kappa_\xi : \xi < \cf(\kappa)\} = \kappa$.

Suppose, toward a contradiction, that $\pi w(X(\kappa)) < \kappa^+$. By the lower bound in Proposition~\ref{prop_FCT_pi_peso_cotas}, it follows that $\pi w(X(\kappa)) = \kappa$, and hence there exists a cofinal subset $\mathcal{A}$ of $[X]^{<\kappa}$ of cardinality $\kappa$.

Let $\{A_\alpha : \alpha<\kappa\}$ be an enumeration without repetitions of $\mathcal{A}$. Moreover, for each $\xi < \cf(\kappa)$, let
\[\textstyle
I_\xi := \{\alpha < \kappa : \alpha < \kappa_\xi\}, \
J_\xi := \{\alpha<\kappa : |A_\alpha| < \kappa_\xi\} \ \text{and} \
B_\xi := \bigcup\{A_\alpha : \alpha \in I_\xi \cap J_\xi\}.
\]
Clearly, $\{B_\xi : \xi<\cf(\kappa)\}$ is a subset of $[X]^{<\kappa}$.

Let $B\in [X]^{<\kappa}$, and let $\alpha<\kappa$ be such that $B \subseteq A_\alpha$. Also, let $\eta<\cf(\kappa)$ be such that $\alpha < \kappa_\eta$, and choose $\eta < \xi < \cf(\kappa)$ such that $|A_{\alpha}| < \kappa_\xi$. Thus, $\alpha \in I_\xi \cap J_\xi$ implies that $A_\alpha \subseteq B_\xi$ and, consequently, $B\subseteq B_\xi$.

Therefore, $\{B_\xi : \xi<\cf(\kappa)\}$ is a cofinal subset of $[X]^{<\kappa}$ of cardinality strictly less than $\pi w(X)$, which is impossible.
\end{proof}

It is well known that the Generalized Continuum Hypothesis (abbreviated as \textsf{GCH} hereafter) allows the infinite cardinal arithmetic to be completely determined (see \cite{kunen1980}). In particular, if $\omega \leq \lambda \leq \kappa$, then, under \textsf{GCH}, we have

\begin{equation}\label{eq_GCH}
\kappa^{<\lambda}
=
\begin{cases}
\kappa, & \text{if } \lambda \leq \cf(\kappa), \\
\kappa^+, & \text{if } \cf(\kappa) < \lambda.
\end{cases}
\end{equation}

In view of~(\ref{eq_GCH}), we are in a position to establish Proposition~\ref{prop_GCH_basico}, with which we conclude this section. Its proof follows immediately from Propositions~\ref{prop_FCT_pi_peso_cotas}--\ref{prop_FCT_pi_peso_singular}.

\begin{proposition}\label{prop_GCH_basico} If \textsf{GCH} holds, then
\[
\pi w(X(\lambda,\kappa)) =
\begin{cases}
\kappa, & \text{if } \lambda \leq \cf(\kappa), \\
\kappa^+, & \text{if } \cf(\kappa) < \lambda.
\end{cases}
\]

\end{proposition}

Consequently, Proposition~\ref{prop_GCH_basico} determines the cofinality of the ideal $[\kappa]^{<\lambda}$ whenever \textsf{GCH} holds.

%%%%%%%%%%%%%%%%%%%%%%%%%%%%%%%%%%%%%%%%%%%%%%%%%%%%%%%%%%%%%%
\section{Chain Conditions}\label{secc_calibres}
%%%%%%%%%%%%%%%%%%%%%%%%%%%%%%%%%%%%%%%%%%%%%%%%%%%%%%%%%%%%%%

The purpose of this section is to study the chain conditions of the spaces $X(\lambda,\kappa)$ described above. A first remark is that, since all spaces of this form are hyperconnected, we have $\P(X(\lambda,\kappa)) = \CN = \WP(X(\lambda,\kappa))$, by Proposition~\ref{prop_basico}(\ref{basico_hiperconexo}). Therefore, the problem of interest is to study the calibers of the spaces $X(\lambda,\kappa)$. In this regard, the present investigation continues part of the work carried out in \cite{riotam2023}. In that article, according to items~(\ref{basico_T1}) and~(\ref{basico_cofinito}) of Proposition~\ref{prop_basico}, it is shown that $\C(X(\omega)) = \UC$ and $\C(X(\omega,\kappa)) = \CN$ whenever $\kappa > \omega$.

Remark~\ref{rmk_calibres_fundamental} is a fundamental tool for the development of this section. The equivalence established there follows immediately by taking complements appropriately with respect to the elements of the topology $\tau(\lambda,\kappa)$.

\begin{remark}\label{rmk_calibres_fundamental}
If $\mu\in\CN$, then $\mu$ is a caliber for the space $X(\lambda,\kappa)$ if and only if, for every family $\{A_\alpha:\alpha<\mu\}\subseteq[X]^{<\lambda}$, there exists a set $J\in[\mu]^\mu$ such that $\bigcup_{\alpha\in J}A_\alpha\neq X$.

\end{remark}

Our first task is to prove Proposition~\ref{prop_calibres_basico}, in which we use the information obtained in the preceding section concerning the cardinal functions of the spaces $X(\lambda,\kappa)$ to obtain some of their calibers.

\begin{proposition}\label{prop_calibres_basico} The following statements hold.

\begin{enumerate}
\item If $\lambda < \kappa$, then every element of the union
\[
\{\mu \in \CN : \mu < \kappa\}
\cup
\{\mu \in \CN : \cf(\mu) > \lambda\}
\]
is a caliber for $X(\lambda,\kappa)$.

\item Every element of the union
\[
\{\mu \in \CN : \mu < \cf(\kappa)\}
\cup
\{\mu \in \CN : \cf(\mu) > \kappa\}
\]
is a caliber for $X(\kappa)$. Moreover,
\[
\{\mu \in \CN : \mu > \pi w(X(\kappa))\ \wedge \ \cf(\kappa) > \cf(\mu)\}
\subseteq
\C(X(\kappa)).
\]

\end{enumerate}

\end{proposition}

\begin{proof} For item~(1), let $A := \{\mu \in \CN : \mu < \kappa\}$ and $B := \{\mu \in \CN : \cf(\mu) > \lambda\}$. First, if $\mu \in A$ and $\{U_\alpha : \alpha<\mu\} \subseteq \tau_X^+$, then the fact that $\mu < \psi(X)$ (see Proposition~\ref{prop_FCT_pseudocaracter_2}) guarantees that $\bigcap_{\alpha<\mu} U_\alpha \neq \emptyset$. Thus, $\mu \in \C(X)$.

On the other hand, if $\mu \in B$, the relation $\cf(\mu) > d(X)$ (see Proposition~\ref{prop_FCT_densidad}) implies, by Proposition~\ref{prop_basico}(\ref{basico_densidad}), that $\mu \in \C(X)$.

As for item~(2), the proof of the first part is analogous to that of item~(1), while the second part follows from Proposition~\ref{prop_basico}(\ref{basico_pi_peso}) and Proposition~\ref{prop_FCT_pseudocaracter_2}.
\end{proof}

We next analyze what happens with the spaces $X(\kappa)$. To begin with, Lemma~\ref{lema_calibres_cof_no_es} provides a sufficient condition for an infinite cardinal not to be a caliber for spaces of this type.

\begin{lemma}\label{lema_calibres_cof_no_es} If $\mu\in \CN$ satisfies $\cf(\mu) = \cf(\kappa)$, then $\mu$ is not a caliber for $X(\kappa)$.

\end{lemma}

\begin{proof} By Proposition~\ref{prop_basico}(\ref{basico_cofinalidad}), to show that $\mu$ is not a caliber for $X$, it suffices to verify that $\nu := \cf(\kappa)$ is not a caliber for $X$.

Let $\{x_\alpha : \alpha < \kappa\}$ be an enumeration without repetitions of $X$, and let $f \colon \nu \to \kappa$ be a strictly increasing cofinal function. Moreover, for each $\xi < \nu$, let $A_\xi := \{x_\alpha : \alpha < f(\xi)\}$. Clearly, $\{A_\xi : \xi < \nu\} \subseteq [X]^{<\kappa}$. Furthermore, for every $J \in [\nu]^\nu$, we have $\bigcup_{\xi \in J} A_\xi = X$. Thus, Remark~\ref{rmk_calibres_fundamental} implies that $\nu$ is not a caliber for $X$ and, consequently, that $\mu$ is not a caliber either.
\end{proof}

With the preceding lemma, we are in a position to compute the calibers of $X(\kappa)$ when $\kappa$ coincides with its cofinality.

\begin{theorem}\label{thm_calibres_lambda_igual_kappa_regular}
If $\kappa$ is regular, then $\C(X(\kappa)) = \{\mu \in \CN : \cf(\mu) \neq \kappa\}$.

\end{theorem}

\begin{proof} Let
\begin{alignat*}{2}
A_0 &:= \{\mu \in \CN : \mu < \kappa\},
&\quad
A_1 &:= \{\mu \in \CN : \cf(\mu) > \kappa\},\\
A_2 &:= \{\mu \in \CN : \mu > \kappa > \cf(\mu)\},
&\quad
A_3 &:= \{\mu \in \CN : \cf(\mu) = \kappa\}.
\end{alignat*}

Observe that $A_3 \cap \bigcup_{i<3} A_i = \emptyset$ and that $\bigcup_{i<4} A_i = \CN$. Moreover, by Proposition~\ref{prop_FCT_kappa_regular}, Proposition~\ref{prop_calibres_basico}(2), and Lemma~\ref{lema_calibres_cof_no_es}, we have
\[
\textstyle
\C(X) \subseteq \CN \setminus A_3 = \bigcup_{i<3} A_i \subseteq \C(X).
\]
Therefore, $\C(X) = \CN \setminus A_3$.
\end{proof}

A pertinent remark is that, by taking $\kappa = \omega$ in Theorem~\ref{thm_calibres_lambda_igual_kappa_regular}, we obtain
\[
\C(X(\omega)) = \{\mu \in \CN : \cf(\mu) \neq \omega\} = \{\mu \in \CN : \cf(\mu) > \omega\} = \UC;
\]
that is, Theorem~\ref{thm_calibres_lambda_igual_kappa_regular} recovers the equality $\C(X(\omega)) = \UC$ obtained in \cite{riotam2023}.

In the case of the space $X(\kappa)$ when $\kappa$ is singular, our first objective is to establish, using Remark~\ref{rmk_calibres_fundamental}, Lemmas~\ref{lema_calibres_singular_1} and~\ref{lema_calibres_singular_2}, which constitute the main tools for this part of the analysis.

\begin{lemma}\label{lema_calibres_singular_1} If $\mu \in \CN$ satisfies $\cf(\kappa) < \cf(\mu) < \kappa$, then $\mu$ is a caliber for $X(\kappa)$.

\end{lemma}

\begin{proof} Let $\{\kappa_\xi : \xi < \cf(\kappa)\}$ be a strictly increasing sequence of regular cardinals such that $\kappa_0 = \cf(\mu)^+$ and $\sup\{\kappa_\xi : \xi < \cf(\kappa)\} = \kappa$. Moreover, let $\{A_\alpha : \alpha < \mu\}$ be a subset of $[X]^{<\kappa}$, and for each $\xi < \cf(\kappa)$, let $J_\xi := \{\alpha < \mu : |A_\alpha| < \kappa_\xi\}$.

Observe that $\mu = \bigcup_{\xi < \cf(\kappa)} J_\xi$. On the other hand, if $|J_\xi| < \mu$ for every $\xi < \cf(\kappa)$, then, since $\cf(\kappa) < \cf(\mu)$, we would have $|\bigcup_{\xi < \cf(\kappa)} J_\xi| < \mu$, contradicting the equality above.

Thus, there exists $\xi < \cf(\kappa)$ such that $|J_\xi| = \mu$. Now, if $Y \in [X]^{\kappa_\xi}$, then, since $X(\kappa_\xi)$ has caliber $\mu$, as $\cf(\mu) < \kappa_0 \leq \kappa_\xi$ (see Theorem~\ref{thm_calibres_lambda_igual_kappa_regular}), and $\{A_\alpha \cap Y : \alpha \in J_\xi\}$ is a subset of $[Y]^{<\kappa_\xi}$, there exists $J \in [J_\xi]^{\mu}$ such that
$\bigcup_{\alpha \in J} (A_\alpha \cap Y) \neq Y$, which implies that $\bigcup_{\alpha \in J} A_\alpha \neq X$.
\end{proof}

\begin{lemma}\label{lema_calibres_singular_2} If $\mu \in \CN$ satisfies $\cf(\mu) < \cf(\kappa) < \mu$, then $\mu$ is a caliber for $X(\kappa)$.

\end{lemma}

\begin{proof} Let $\{A_\alpha : \alpha<\mu\}$ be a subset of $[X]^{<\kappa}$. Moreover, let $\{\kappa_\xi : \xi < \cf(\kappa)\}$ and $\{\mu_\eta : \eta < \cf(\mu)\}$ be strictly increasing sequences of regular cardinals such that $\kappa_0 = \cf(\kappa)$, $\sup\{\kappa_\xi : \xi < \cf(\kappa)\} = \kappa$, and $\sup\{\mu_\eta : \eta < \cf(\mu)\} = \mu$. Furthermore, for each $\xi < \cf(\kappa)$, let $J_\xi := \{\alpha < \mu : |A_\alpha| < \kappa_\xi\}$. Note that $\mu = \bigcup_{\xi \in I} J_\xi$ for every $I \in [\cf(\kappa)]^{\cf(\kappa)}$.

\medskip

\noindent {\bf Claim.} There exists $\xi < \cf(\kappa)$ such that $|J_\xi| = \mu$.

\medskip

Suppose, toward a contradiction, that $|J_\xi| < \mu$ for every $\xi<\cf(\kappa)$. Let $f : \cf(\kappa) \to \cf(\mu)$ be the function defined by
\[
f(\xi) := \min\{\eta < \cf(\mu) : |J_\xi| < \mu_\eta\}.
\]
The relation $\cf(\mu) < \cf(\kappa)$, together with the regularity of the cardinal $\cf(\kappa)$, yields some $\eta < \cf(\mu)$ and $I \in [\cf(\kappa)]^{\cf(\kappa)}$ such that $f[I] \subseteq \{\eta\}$. Thus, we obtain the contradiction
\[
\textstyle 
\mu = |\bigcup_{\xi \in I} J_\xi| \leq |I| \cdot \sup\{|J_\xi| : \xi \in I\} \leq \cf(\kappa) \cdot \mu_\eta < \mu.
\]

\medskip

Finally, if $\xi < \cf(\kappa)$ satisfies $|J_\xi| = \mu$ and $Y \in [X]^{\kappa_\xi}$, then, since $X(\kappa_\xi)$ has caliber $\mu$, as $\cf(\mu) < \kappa_0 \leq \kappa_\xi$ (see Theorem~\ref{thm_calibres_lambda_igual_kappa_regular}), and $\{A_\alpha \cap Y : \alpha \in J_\xi\}$ is a subset of $[Y]^{<\kappa_\xi}$, there exists a set $J \in [J_\xi]^{\mu}$ such that
$\bigcup_{\alpha \in J} (A_\alpha \cap Y) \neq Y$; in particular, $\bigcup_{\alpha \in J} A_\alpha \neq X$.
\end{proof}

Based on the preceding results, it is possible to determine the calibers of $X(\kappa)$ when the cardinal $\kappa$ is singular.

\begin{theorem}\label{thm_calibres_lambda_igual_kappa_singular} If $\kappa$ is singular, then $\C(X(\kappa)) = \{\mu \in \CN : \cf(\mu) \neq \cf(\kappa)\}$.

\end{theorem}

\begin{proof} Let
\begin{align*}
A_0 &:= \{\mu \in \CN : \mu < \cf(\kappa)\}, \\
A_1 &:= \{\mu \in \CN : \cf(\mu) > \kappa\}, \\
A_2 &:= \{\mu \in \CN : \cf(\kappa) < \cf(\mu) < \kappa\}, \\
A_4 &:= \{\mu \in \CN : \cf(\mu) < \cf(\kappa) < \mu\} \hspace{0.5em} \text{and} \\
A_5 &:= \{\mu \in \CN : \cf(\mu) = \cf(\kappa)\}.
\end{align*}
Observe that $A_5 \cap \bigcup_{i<5} A_i = \emptyset$ and that $\bigcup_{i<6} A_i = \CN$. Moreover, by Proposition~\ref{prop_calibres_basico}(2) and Lemmas~\ref{lema_calibres_cof_no_es}, \ref{lema_calibres_singular_1}, and~\ref{lema_calibres_singular_2}, we have
\[
\textstyle
\C(X) \subseteq \CN \setminus A_5 = \bigcup_{i<5} A_i \subseteq \C(X).
\]
Consequently, $\C(X) = \CN \setminus A_5$.
\end{proof}

Finally, we proceed to determine the calibers of $X(\lambda,\kappa)$ under the assumption $\lambda < \kappa$.

\begin{theorem}\label{thm_calibres_lambda_menor_kappa} If $\lambda < \kappa$, then $\C(X(\lambda,\kappa)) = \CN$.

\end{theorem}

\begin{proof} Let $\mu \in \CN$. If $\cf(\mu) > \lambda$, then Proposition~\ref{prop_calibres_basico}(1) implies that $\mu \in \C(X)$. On the other hand, when $\cf(\mu) \leq \lambda < \kappa$, let $\{A_\alpha : \alpha < \mu\} \subseteq [X]^{<\lambda}$, $\nu := \lambda^+$, and $Y \in [X]^{\nu}$. Since $\{A_\alpha \cap Y : \alpha < \mu\}$ is a subset of $[Y]^{<\nu}$ and the space $X(\nu)$ has caliber $\mu$, as $\cf(\mu) \leq \lambda < \nu$ (see Theorem~\ref{thm_calibres_lambda_igual_kappa_regular}), it follows that there exists $J \in [\mu]^{\mu}$ such that $\bigcup_{\alpha \in J} (A_\alpha \cap Y) \neq Y$, which implies that $\bigcup_{\alpha \in J} A_\alpha \neq X$. Therefore, $\mu$ is a caliber for $X$, by Remark~\ref{rmk_calibres_fundamental}.
\end{proof}

An immediate consequence of Theorem~\ref{thm_calibres_lambda_menor_kappa} is obtained by considering the case $\kappa > \omega$: in this situation, we obtain $\C(X(\omega,\kappa)) = \CN$. In other words, this theorem generalizes and, in particular, recovers the equality $\C(X(\omega,\kappa)) = \CN$ established in \cite{riotam2023}.

The main results of this section, namely, Theorems~\ref{thm_calibres_lambda_igual_kappa_regular}, \ref{thm_calibres_lambda_igual_kappa_singular}, and~\ref{thm_calibres_lambda_menor_kappa}, are summarized in the following result, which concludes this section.

\begin{theorem}
If $\lambda < \kappa$, then
\[
\C(X(\kappa)) = \{\mu \in \CN : \cf(\mu) \neq \cf(\kappa)\}
\quad \text{and} \quad
\C(X(\lambda,\kappa)) = \CN.
\]

\end{theorem}

%%%%%%%%%%%%%%%%%%%%%%%%%%%%%%%%%%%%%%%%%%%%%%%%%%%%%%%%%%%%%%
\section{Topological Hyperspaces}
%%%%%%%%%%%%%%%%%%%%%%%%%%%%%%%%%%%%%%%%%%%%%%%%%%%%%%%%%%%%%%

Basic material on hyperspaces can be found in \cite{ginsburg1975} and \cite{michael1951}. We now establish the terminology and notation needed to provide the appropriate framework and continue the development of the article.

Let $X$ be a topological space,
\begin{align*}
\CL(X) &:= \big\{ A \subseteq X : A \text{ is a nonempty closed subset of } X \big\} \hspace{0.5em} \text{and} \\
\F(X) &:= \big\{ A \in \CL(X) : A \text{ is finite} \big\}.
\end{align*}
Moreover, for each $\mathcal{U}\subseteq \Pot(X)$, let
\[\textstyle
\vt{\mathcal{U}} := \big\{ A \in \CL(X) : A \subseteq \bigcup \mathcal{U} \ \wedge \ \forall \, U \in \mathcal{U}\,(A \cap U \neq \emptyset) \big\}.
\]
The \emph{Vietoris topology} is the topology on $\CL(X)$ generated by the base
\[
\textstyle
\big\{ \vt{\mathcal{U}} : \mathcal{U} \in [\tau_X^+]^{<\omega} \big\}.
\]

For notational convenience, if $n \in \mathbb{N}$ and $U_1, \dots, U_n \in \tau_X$, we write
$\vt{U_1,\dots,U_n}$ instead of $\vt{\{U_1,\dots,U_n\}}$. For example, if $U,V \in \tau_X$, then
\[
\textstyle
\vt{U,V} = \big\{ A \in \CL(X) : A \subseteq U \cup V \ \wedge \ A \cap U \neq \emptyset \ \wedge \ A \cap V \neq \emptyset \big\}.
\]
Moreover, when referring to a hyperspace of $X$, say $\HH(X)$, it is understood that $\HH(X)$ is a subset of $\CL(X)$ equipped with the subspace topology inherited from the latter.

From this point onward, we use the notation $\vt{U_1,\dots,U_n}_{\HH}$ to denote the basic open set $\vt{U_1,\dots,U_n}$ intersected with an arbitrary hyperspace $\HH(X)$; that is,
\[
\vt{U_1,\dots,U_n}_{\HH} := \vt{U_1,\dots,U_n} \cap \HH(X).
\]
In particular, the symbols $\vt{U_1,\dots,U_n}_{\CL}$ and $\vt{U_1,\dots,U_n}$ denote the same set.

We now analyze the chain conditions in Vietoris hyperspaces. It is important to keep in mind that throughout this section, $X$ is assumed to be a $T_1$ topological space and $\HH(X)$ is a hyperspace of $X$ intermediate between $\F(X)$ and $\CL(X)$.

To begin, a well-known result of Michael, which states that $\F(X)$ is dense in $\CL(X)$, immediately yields part (1) of Proposition~\ref{prop_hip_Michael_Ginsburg}. On the other hand, \cite[Theorem~2.7, p.~78]{ginsburg1975} proves that the $\pi$-weight of $X$ coincides with that of $\CL(X)$. Using the same technique, one can obtain part (2) of the following result.

\begin{proposition}\label{prop_hip_Michael_Ginsburg} The following statements hold.

\begin{enumerate}
\item $\F(X)$ is dense in $\HH(X)$, and $\HH(X)$ is dense in $\CL(X)$.

\item $\pi w(\HH(X)) = \pi w(X)$.

\end{enumerate}

\end{proposition}

With regard to the chain conditions of the hyperspace $\HH(X)$, it is worth noting that, although Propositions~\ref{prop_hip_calibres_CL_X} and~\ref{prop_hip_calibres_UC_X_F} are stated for calibers, precalibers, and weak precalibers, in view of the similarity of the arguments, we provide details only for calibers.

\begin{proposition}\label{prop_hip_calibres_CL_X} If $\kappa$ is a caliber (resp., precaliber; weak precaliber) for $\CL(X)$, then $\kappa$ is a caliber (resp., precaliber; weak precaliber) for $X$.

\end{proposition}

\begin{proof} Let $\kappa$ be a caliber for $\CL(X)$, and let $\{U_\alpha : \alpha<\kappa\}$ be a subset of $\tau_X^+$. Since $\{\vt{U_\alpha} : \alpha<\kappa\}$ is a subset of $\tau_{\CL(X)}^{+}$, there exists $J \in [\kappa]^{\kappa}$ such that $\{\vt{U_\alpha} : \alpha\in J\}$ has nonempty intersection. Thus, the condition $\bigcap_{\alpha \in J} \vt{U_\alpha} \neq \emptyset$ forces $\bigcap_{\alpha \in J} U_\alpha \neq \emptyset$. Consequently, $\kappa$ is a caliber for $X$.
\end{proof}

An immediate consequence of Propositions~\ref{prop_basico}(\ref{basico_denso}) and~\ref{prop_hip_Michael_Ginsburg}(1) is the following:

\begin{proposition}\label{prop_hip_relaciones} The following relations hold:
\begin{align*}
\C(\F(X))&\subseteq \C(\HH(X))\subseteq \C(\CL(X)), \\
\P(\F(X))&=\P(\HH(X))=\P(\CL(X)) \hspace{0.5em} \text{and} \\
\WP(\F(X))&=\WP(\HH(X))=\WP(\CL(X)).
\end{align*}

\end{proposition}

Proposition~\ref{prop_hip_calibres_UC_X_F} aims to show that a broad class of chain conditions on $X$ always ``lifts'' to the hyperspace $\F(X)$:

\begin{proposition}\label{prop_hip_calibres_UC_X_F} Let $\kappa$ be a cardinal such that $\cf(\kappa)>\omega$. If $\kappa$ is a caliber (resp., precaliber; weak precaliber) for $X$, then $\kappa$ is a caliber (resp., precaliber; weak precaliber) for $\F(X)$.

\end{proposition}

\begin{proof} Let $\{U_\alpha : \alpha<\kappa\}$ be a collection of basic open sets of $\F(X)$. For each $n\in \mathbb{N}$, consider the collection
\[\textstyle
J_n := \big\{\alpha<\kappa : \exists V_1,\dots V_n\in \tau_X^+\big(U_\alpha = \vt{V_1,\dots V_n}_{\F}\big)\big\}.
\]
Since $\kappa = \bigcup_{n \in \mathbb{N}} J_n$, the relation $\cf(\kappa)>\omega$ yields some $m\in \mathbb{N}$ such that $|J_m| = \kappa$. For each $\alpha \in J_m$, let $V(\alpha,1),\dots,V(\alpha,m) \in \tau_X^+$ be such that \[
U_\alpha = \vt{V(\alpha,1),\dots,V(\alpha,m)}_{\F}.
\] Also, let
\[\mathcal{V} := \{V(\alpha,1)\times \dots \times V(\alpha,m) : \alpha\in J_m\}.
\]

Now, since $\kappa$ is a caliber for $X$, Proposition~\ref{prop_basico}(\ref{basico_producto}) implies that $\kappa$ is a caliber for $X^m$. Thus, since $\mathcal{V}$ is a subset of $\tau_{X^m}^+$, there exists $J\in [J_m]^{\kappa}$ such that the family $\left\{V(\alpha,1)\times \dots \times V(\alpha,m) : \alpha\in J\right\}$ has a nonempty intersection. Therefore, if $(x_1,\dots,x_n)\in \bigcap \{V(\alpha,1)\times \dots \times V(\alpha,m) : \alpha\in J\}$, it follows that $\{x_1,\dots,x_n\} \in \bigcap_{\alpha \in J} U_\alpha$. Consequently, $\kappa$ is a caliber for $\F(X)$.
\end{proof}

The results established in Propositions~\ref{prop_hip_calibres_CL_X}, \ref{prop_hip_relaciones}, and~\ref{prop_hip_calibres_UC_X_F} can be represented graphically in the diagram in Figure~\ref{figura_hiperespacios}. The relation $A \to B$ means that $A \subseteq B$:

\begin{figure}[H]
\begin{center}
\begin{tikzcd}[scale cd=0.85]

\C(X) \cap \UC \arrow[d]{}{} \arrow[r]{}{} & \C(\F(X)) \arrow[d]{}{} \arrow[r]{}{} & \C(\HH(X)) \arrow[d]{}{} \arrow[r]{}{} & \C(\CL(X)) \arrow[d]{}{} \arrow[r]{}{} & \C(X) \arrow[d]{}{} \\

\P(X) \cap \UC \arrow[d]{}{} \arrow[r]{}{} & \P(\F(X)) \arrow[d]{}{} \arrow[r,leftrightarrow]{}{} & \P(\HH(X)) \arrow[d]{}{} \arrow[r,leftrightarrow]{}{} & \P(\CL(X)) \arrow[d]{}{} \arrow[r]{}{} & \P(X) \arrow[d]{}{} \\

\WP(X) \cap \UC \arrow[r]{}{} & \WP(\F(X))  \arrow[r,leftrightarrow]{}{} & \WP(\HH(X))  \arrow[r,leftrightarrow]{}{} & \WP(\CL(X))  \arrow[r]{}{} & \WP(X) \\

\end{tikzcd}
\end{center}
\vspace{-10mm}
\caption{Chain conditions in Vietoris hyperspaces.}\label{figura_hiperespacios}
\end{figure}
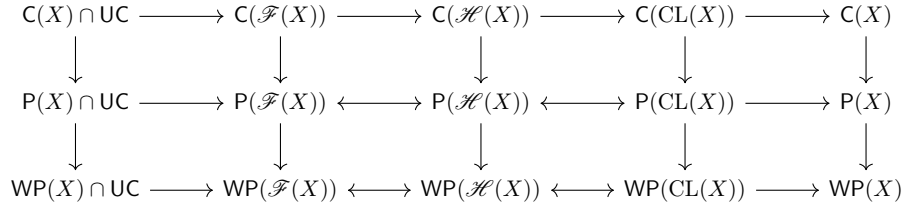

We summarize in a single result the information contained in the first row of the preceding diagram, so that we can refer to it later.

\begin{corollary}\label{cor_hip_primer_renglon}
The following relations hold:
\[
\C(X) \cap \UC \subseteq \C(\F(X)) \subseteq \C(\HH(X)) \subseteq \C(\CL(X)) \subseteq \C(X).
\]
\end{corollary}

The information presented in this section, summarized in the diagram in Figure~\ref{figura_hiperespacios}, allows us to completely determine the chain conditions of any hyperspace intermediate between $\F(X)$ and $\CL(X)$, provided that $X$ is an infinite $T_2$ space. Indeed, it suffices to observe that if $X$ is infinite and Hausdorff, then Proposition~\ref{prop_basico}(\ref{basico_T2}) guarantees that $\WP(X) \subseteq \UC$. Consequently, we obtain the following theorem.

\begin{theorem}\label{thm_hip_condiciones_cadena_H}
If $X$ is an infinite $T_2$ space, then
\[
\C(X) = \C(\HH(X)), \quad
\P(X) = \P(\HH(X)) \quad \text{and} \quad
\WP(X) = \WP(\HH(X)).
\]

\end{theorem}

%%%%%%%%%%%%%%%%%%%%%%%%%%%%%%%%%%%%%%%%%%%%%%%%%%%%%%%%%%%%%%
\section{Hyperspaces over Hyperconnected Spaces}
%%%%%%%%%%%%%%%%%%%%%%%%%%%%%%%%%%%%%%%%%%%%%%%%%%%%%%%%%%%%%%

The Hausdorff assumption in Theorem~\ref{thm_hip_condiciones_cadena_H} leaves open the question of what happens when the space $X$ is not $T_2$. The goal of this section is to show that, in the context of canonical hyperconnected spaces, an analogous result to that theorem can be obtained even without the aforementioned separation property (see Theorem~\ref{thm_hiphip_calibres_resumen}, stated below).

First, since the proof of Lemma~\ref{lema_hiphip_denso_hiperconexo} is elementary, we omit it, but we state the result in an appropriate setting for later use.

\begin{lemma}\label{lema_hiphip_denso_hiperconexo} The following statements are equivalent for every topological space $X$.

\begin{enumerate}
\item $X$ is hyperconnected.
\item Every dense subspace of $X$ is hyperconnected.
\item There exists a dense subspace of $X$ that is hyperconnected.
\end{enumerate}

\end{lemma}

From Proposition~\ref{prop_hiphip_hip_hip} through Theorem~\ref{thm_hiphip_infinito_numerable}, $X$ is assumed to be a $T_1$ topological space and $\HH(X)$ is a hyperspace of $X$ intermediate between $\F(X)$ and $\CL(X)$.

\begin{proposition}\label{prop_hiphip_hip_hip} If $X$ is a hyperconnected space, then $\HH(X)$ is hyperconnected.

\end{proposition}

\begin{proof} By Proposition~\ref{prop_hip_Michael_Ginsburg}(1) and Lemma~\ref{lema_hiphip_denso_hiperconexo}, it suffices to show that $\F(X)$ is hyperconnected. If $\mathcal{U},\mathcal{V} \in [\tau_X^+]^{<\omega}\setminus\{\emptyset\}$, then, since $\mathcal{U}\cup\mathcal{V} \subseteq \tau_X^+$ and $X$ is hyperconnected, there exists a point $x\in \bigcap\left(\mathcal{U}\cup\mathcal{V}\right)$; in particular, $\{x\}$ is an element of $\vt{\mathcal{U}}_{\F} \cap \vt{\mathcal{V}}_{\F}$.
\end{proof}

Thus, a combination of Propositions~\ref{prop_basico}(\ref{basico_hiperconexo}) and~\ref{prop_hiphip_hip_hip} yields the following:

\begin{corollary}\label{cor_hiphip_hip_hip} If $X$ is a hyperconnected space, then
\[
\P(\HH(X)) = \WP(\HH(X)) = \P(X) = \WP(X) = \CN.
\]

\end{corollary}

On the other hand, when $X$ is countably infinite, a relation similar to those stated in Corollary~\ref{cor_hiphip_hip_hip} can be obtained in the case of calibres. Indeed, by Proposition~\ref{prop_basico}(\ref{basico_T1}), every infinite countable $T_1$ space satisfies
\[
\C(X)=\UC.
\]
The statement below follows from this fact and Corollary~\ref{cor_hip_primer_renglon}.

\begin{theorem}\label{thm_hiphip_infinito_numerable} If $X$ is countably infinite, then
\[
\C(\HH(X)) = \C(X) = \UC.
\]

\end{theorem}

Our next goal is to study the chain conditions of hyperspaces over canonical hyperconnected spaces, taking advantage of the substantial information about them obtained in Sections~\ref{secc_FCT} and~\ref{secc_calibres}.

From now on, unless explicitly stated otherwise, $\kappa$ and $\lambda$ denote infinite cardinal numbers such that $\lambda \leq \kappa$. In addition, we assume that the hyperspace $\HH(X(\lambda,\kappa))$ satisfies
\[
\F(X(\lambda,\kappa)) \subseteq \HH(X(\lambda,\kappa)) \subseteq \CL(X(\lambda,\kappa)).
\]

By Corollary~\ref{cor_hiphip_hip_hip}, we have
\[
\P(\HH(X(\lambda,\kappa))) = \WP(\HH(X(\lambda,\kappa))) = \P(X(\lambda,\kappa)) = \WP(X(\lambda,\kappa)) = \CN.
\]
Therefore, once again, the problem of interest is to calculate the calibers of the hyperspaces of canonical hyperconnected spaces.

The following preliminary result, which is closely related to the pseudocharacter of $X(\lambda,\kappa)$ and Proposition~\ref{prop_FCT_pseudocaracter_1}, is established in Proposition~\ref{prop_hiphip_pseudocaracter}.

\begin{proposition}\label{prop_hiphip_pseudocaracter} The following relation holds:
\[\textstyle
\psi(X(\lambda,\kappa)) = \min \big\{|\mathcal{U}| : \mathcal{U} \subseteq \tau_{\HH(X)}^+ \ \wedge \ \bigcap \mathcal{U} = \emptyset\big\}.
\]
\end{proposition}

\begin{proof} Let $\nu := \psi(X)$ and $Y := \HH(X)$. First, observe that if $\mathcal{V} \subseteq \tau_X^+$ satisfies $|\mathcal{V}| = \nu$ and $\bigcap \mathcal{V} = \emptyset$, then $\mathcal{U} := \{\langle V \rangle_{\HH} : V \in \mathcal{V}\}$ satisfies $\mathcal{U} \subseteq \tau_Y^+$, $|\mathcal{U}| = \nu$, and $\bigcap \mathcal{U} = \emptyset$.

Now, let $\theta < \nu$ be a positive cardinal, and let $\{U_\alpha : \alpha < \theta\}$ be a collection of basic open sets of $Y$. If $\theta < \omega$, then the fact that $Y$ is hyperconnected (see Proposition~\ref{prop_hiphip_hip_hip}) implies that $\bigcap_{\alpha<\theta} U_\alpha \neq \emptyset$. Otherwise, if $\theta \geq \omega$, for each $\alpha < \theta$, there exist $n_\alpha \in \mathbb{N}$ and $\{U(\alpha,i) : 1 \leq i \leq n_\alpha\} \subseteq \tau_X^+$ such that
$U_\alpha = \langle U(\alpha,1),\ldots,U(\alpha,n_\alpha)\rangle_{\HH}$.
Then, since
\[
|\{U(\alpha,i) : \alpha < \theta \ \wedge\ 1 \leq i \leq n_\alpha\}| \leq \theta < \nu,
\]
we have
\[\textstyle
\bigcap \{U(\alpha,i) : \alpha < \theta \ \wedge\ 1 \leq i \leq n_\alpha\} \neq \emptyset.
\]
Thus, if $x$ is an element of this intersection, then $\{x\} \in \bigcap_{\alpha<\theta} U_\alpha$.
\end{proof}

Proposition~\ref{prop_hiphip_basico} follows from Propositions~\ref{prop_basico}(\ref{basico_pi_peso}), \ref{prop_FCT_pseudocaracter_2}, \ref{prop_hip_Michael_Ginsburg}(2), and~\ref{prop_hiphip_pseudocaracter}.

\begin{proposition}\label{prop_hiphip_basico} Every element of the union
\[
\{\mu \in \CN : \mu < \cf(\kappa)\}
\cup
\{\mu \in \CN : \mu > \pi w(X(\kappa))\ \wedge \ \cf(\kappa) > \cf(\mu)\}
\]
is a caliber for $\HH(X(\kappa))$.

\end{proposition}

For the proof of Theorem~\ref{thm_hiphip_lambda_igual_kappa_regular}, it is important to keep in mind that, if $\kappa$ is a regular cardinal, then $\C(X(\kappa)) = \{\mu \in \CN : \cf(\mu) \neq \kappa\}$ and that the $\pi$-weight of $X(\kappa)$ is equal to $\kappa$ (see Proposition~\ref{prop_FCT_kappa_regular} and Theorem~\ref{thm_calibres_lambda_igual_kappa_regular}).

\begin{theorem}\label{thm_hiphip_lambda_igual_kappa_regular} If $\kappa$ is regular, then
\[
\C(\HH(X(\kappa))) = \{\mu \in \CN : \cf(\mu) \neq \kappa\} = \C(X(\kappa)).
\]

\end{theorem}

\begin{proof} By Corollary~\ref{cor_hip_primer_renglon}, it suffices to prove that $\C(X) \setminus \UC \subseteq \C(\F(X))$. Moreover, since $\C(X) \setminus \UC = \emptyset$ if $\kappa = \omega$, in which case the desired inclusion holds, we may assume throughout the argument that $\kappa > \omega$.

Let $\mu \in \C(X)$ be such that $\cf(\mu) = \omega$. Note that the condition $\mu \in \C(X)$ precludes $\mu = \kappa$. Furthermore, Proposition~\ref{prop_hiphip_basico} implies that $\mu \in \C(\F(X))$.
\end{proof}

Before presenting the result corresponding to Theorem~\ref{thm_hiphip_lambda_igual_kappa_singular}, which addresses the case in which $\kappa$ is a singular cardinal, we state Lemma~\ref{lema_hiphip_lambda_igual_kappa_singular}, which will be a key tool in the proof of Theorem~\ref{thm_hiphip_lambda_igual_kappa_singular}.

\begin{lemma}\label{lema_hiphip_lambda_igual_kappa_singular} If $\mu \in \CN$ satisfies $\cf(\mu) = \omega < \cf(\kappa) < \mu$, then $\mu$ is a caliber for $\F(X(\kappa))$.

\end{lemma}

\begin{proof} Let $\{\kappa_\xi : \xi < \cf(\kappa)\}$ be a strictly increasing sequence of cardinals such that $\kappa_0 = \cf(\kappa)$ and $\sup\{\kappa_\xi : \xi < \cf(\kappa)\} = \kappa$. Furthermore, let $\{U_\alpha : \alpha < \mu\}$ be a subset of $\tau_{\F(X)}^{+}$ consisting of basic open sets and, for each $\alpha < \mu$, let $n_\alpha \in \mathbb{N}$ and $\{U(\alpha,i) : 1\leq i \leq n_\alpha\} \subseteq \tau_X^+$ be such that $U_\alpha = \langle U(\alpha,1),\ldots,U(\alpha,n_\alpha)\rangle_{\F}$.

Moreover, for every $\alpha < \mu$ and $1 \leq i \leq n_\alpha$, let $A(\alpha,i) \in [X]^{<\kappa}$ be such that $U(\alpha,i) = X\setminus A(\alpha,i)$. Also, for each $\alpha < \mu$, let
\[
\textstyle
\xi(\alpha) := \min\big\{\xi < \cf(\kappa) : \forall\, 1 \leq i \leq n_\alpha\ (|A(\alpha,i)| < \kappa_\xi)\big\}
\]
and, for every $\xi < \cf(\kappa)$, let $J_\xi := \{\alpha < \mu : \xi(\alpha) = \xi\}$.

By an argument similar to that used in the Claim of Lemma~\ref{lema_calibres_singular_2}, there exists $\xi < \cf(\kappa)$ such that $|J_\xi|=\mu$. Now let $Y \in [X]^{\kappa_\xi}$, endowed with the topology $\tau(\kappa_\xi)$. Additionally, for every $\alpha \in J_\xi$ and $1 \leq i \leq n_\alpha$, let $B(\alpha,i):=A(\alpha,i)\cap Y$. The inclusion $\{B(\alpha,i):\alpha\in J_\xi \ \wedge\ 1\leq i\leq n_\alpha\}\subseteq [Y]^{<\kappa_\xi}$ implies that, for each $\alpha\in J_\xi$, the set
\[
\textstyle
V_\alpha:=\langle Y\setminus B(\alpha,1),\ldots,Y\setminus B(\alpha,n_\alpha)\rangle_{\F(Y)}
\]
is a basic open set in the hyperspace $\F(Y)$.

Thus, since $\mu$ is a caliber for $\F(X(\kappa_\xi))$, as $\cf(\mu) = \omega < \cf(\kappa) \leq \kappa_\xi$ (see Theorem~\ref{thm_hiphip_lambda_igual_kappa_regular}), $X(\kappa_\xi)$ is homeomorphic to $Y$, and $\{V_\alpha : \alpha \in J_\xi\} \subseteq \tau_{\F(Y)}^{+}$, there exist a set $J\in[J_\xi]^\mu$ and a point $B\in\bigcap_{\alpha\in J}V_\alpha$. Consequently, $J\in[\mu]^\mu$ and $B\in\bigcap_{\alpha\in J}U_\alpha$. Hence, $\F(X)$ has caliber $\mu$.
\end{proof}

For the proof of Theorem~\ref{thm_hiphip_lambda_igual_kappa_singular}, it should be kept in mind that, if $\kappa$ is a singular cardinal, then $\C(X(\kappa)) = \{\mu \in \CN : \cf(\mu) \neq \cf(\kappa)\}$ (see Theorem~\ref{thm_calibres_lambda_igual_kappa_singular}).

\begin{theorem}\label{thm_hiphip_lambda_igual_kappa_singular}
If $\kappa$ is singular, then
\[
\C(\HH(X(\kappa))) = \{\mu \in \CN : \cf(\mu) \neq \cf(\kappa)\} = \C(X(\kappa)).
\]
\end{theorem}

\begin{proof}
By Corollary~\ref{cor_hip_primer_renglon}, it suffices to verify that $\C(X) \setminus \UC \subseteq \C(\F(X))$. Moreover, since the assumption $\cf(\kappa) = \omega$ guarantees that $\C(X) \setminus \UC = \emptyset$, which implies that the desired inclusion holds, we may assume throughout the argument that $\cf(\kappa) > \omega$.

Let $\mu \in \C(X)$ with $\cf(\mu) = \omega$. First, if $\mu < \cf(\kappa)$, Proposition~\ref{prop_hiphip_basico} implies that $\mu$ is a caliber for $\F(X)$. Otherwise, if $\mu \geq \cf(\kappa)$, it cannot be the case that $\mu = \cf(\kappa)$, since $\mu$ has countable cofinality. Therefore, $\cf(\mu) = \omega < \cf(\kappa) < \mu$. Finally, under these inequalities, Lemma~\ref{lema_hiphip_lambda_igual_kappa_singular} allows us to conclude that $\mu$ is a caliber for $\F(X)$.
\end{proof}

For the proof of Theorem~\ref{thm_hiphip_lambda_menor_kappa}, it is useful to recall a couple of classical results concerning Vietoris hyperspaces. First, if $X$ is $T_1$ and $Y$ is a dense subspace of $X$, then $\F(Y)$ is a dense subspace of $\F(X)$. On the other hand, if $Z$ is $T_1$ and a continuous image of $X$, then $\F(Z)$ is a continuous image of $\F(X)$. In addition, it is important to keep in mind that, if $\lambda < \kappa$, then $\C(X(\lambda,\kappa)) = \CN$, by Theorem~\ref{thm_calibres_lambda_menor_kappa}.

\begin{theorem}\label{thm_hiphip_lambda_menor_kappa}
If $\lambda < \kappa$, then
\[
\C(\HH(X(\lambda,\kappa))) = \CN = \C(X(\lambda,\kappa)).
\]
\end{theorem}

\begin{proof}
By Corollary~\ref{cor_hip_primer_renglon}, it suffices to verify that $\C(X) \setminus \UC \subseteq \C(\F(X))$. Observe that if $Y \in [X]^{\lambda^+}$, then $Y$ is a dense subspace of $X$ that inherits the topology $\tau(\lambda,\lambda^+)$. On the other hand, since the identity map $X(\lambda^+) \to X(\lambda,\lambda^+)$ is continuous and surjective, it follows that $\F(X(\lambda,\lambda^+))$ is a continuous image of $\F(X(\lambda^+))$. Thus, items~(\ref{basico_denso}) and~(\ref{basico_funcion}) of Proposition~\ref{prop_basico} imply that
\[
\C(\F(X(\lambda^+))) \subseteq \C(\F(X(\lambda,\lambda^+))) = \C(\F(Y)) \subseteq \C(\F(X)).
\]
Finally, if $\mu \in \CN$ satisfies $\cf(\mu) = \omega$, then, since $\F(X(\lambda^+))$ has caliber $\mu$, as $\cf(\mu) = \omega \leq \lambda < \lambda^+$ (see Theorem~\ref{thm_hiphip_lambda_igual_kappa_regular}), it follows that $\mu \in \C(\F(X))$.
\end{proof}

The following theorem summarizes the main results obtained in this section, corresponding to Theorems~\ref{thm_hiphip_lambda_igual_kappa_regular}, \ref{thm_hiphip_lambda_igual_kappa_singular}, and~\ref{thm_hiphip_lambda_menor_kappa}:

\begin{theorem}\label{thm_hiphip_calibres_resumen} If $\lambda<\kappa$, then
\begin{align*}
\C(\HH(X(\kappa))) &= \{\mu \in \CN : \cf(\mu) \neq \cf(\kappa)\} = \C(X(\kappa)) \hspace{0.5em} \text{and} \\
\C(\HH(X(\lambda,\kappa))) &= \CN = \C(X(\lambda,\kappa)).
\end{align*}

\end{theorem}

Canonical hyperconnected spaces form a particular class of $T_1$ spaces that are never $T_2$ and have the property that the calibers of the hyperspace $\HH$ coincide with those of the underlying space. Moreover, Theorems~\ref{thm_hip_condiciones_cadena_H} and~\ref{thm_hiphip_infinito_numerable} show that this coincidence of calibers also occurs in other particular cases related to separation and cardinality properties. However, it is still unknown whether this condition holds for all $T_1$ spaces, even when they are hyperconnected. This situation raises the following questions, whose answers are currently unknown:

\begin{question} Let $X$ be a non-$T_2$, uncountable $T_1$ topological space. Also, let $\HH(X)$ be a hyperspace intermediate between $\F(X)$ and $\CL(X)$.

\begin{enumerate}[label=(\arabic*)]

\item Is it always the case that $\C(X)=\C(\HH(X))$ whenever $X$ is hyperconnected?

\item In general, if $X$ satisfies the properties specified in the hypotheses, does
\[
\C(X)=\C(\HH(X))?
\]

\item Does there exist a space $X$ satisfying the properties specified in the hypotheses such that
\[
\C(\F(X))\neq\C(\K(X)), \hspace{0.5em} \C(\K(X))\neq\C(\CL(X)) \hspace{0.5em} \text{or} \hspace{0.5em} \C(\F(X))\neq\C(\CL(X))?
\]

\end{enumerate}

\end{question}

\end{document}